\documentclass[12pt, a4paper]{amsart}

\usepackage[hmargin=30mm, vmargin=25mm, includefoot, twoside]{geometry}
\usepackage[bookmarksopen=true]{hyperref}

\usepackage{amsfonts,amssymb,verbatim}
\usepackage{latexsym}
\usepackage{mathrsfs}
\usepackage{stmaryrd}
\usepackage{xspace}
\usepackage{enumerate, paralist}
\usepackage{graphicx}
\usepackage[all]{xy}
\usepackage{extarrows}
\usepackage{tabu}
\usepackage{tikz-cd}

\usepackage{txfonts, pxfonts}

\usepackage{amsthm}
\usepackage{amsmath}

\newtheorem{thm}{Theorem}[section]
 
 \newtheorem{lem}[thm]{Lemma}

\newtheorem{alphthm}{Theorem}			

 \theoremstyle{definition}
  \newtheorem{defn}[thm]{Definition}

 \theoremstyle{remark}

   \newtheorem{claim}[thm]{Claim}
\newtheorem*{claim*}{Claim}

\def\NN{\mathbb N}
\def\ZZ{\mathbb Z}
\def\diam{\mathrm{diam}}
\def\T{\mathcal{T}}

\DeclareMathOperator{\supp}{supp}
\newcommand{\norm}[1]{\|#1\|}

\begin{document}

\title{Building weight-free F\o{}lner sets for Yu's Property A in coarse geometry}

\author{Graham A. Niblo, Nick Wright and Jiawen Zhang}

\address[Graham A. Niblo]{School of Mathematical Sciences, University of Southampton, Highfield, SO17 1BJ, United Kingdom.}
\email{g.a.niblo@soton.ac.uk}

\address[Nick Wright]{School of Mathematical Sciences, University of Southampton, Highfield, SO17 1BJ, United Kingdom.}
\email{n.j.wright@soton.ac.uk}

\address[Jiawen Zhang]{School of Mathematical Sciences, Fudan University, 220 Handan Road, Shanghai, 200433, China.}
\email{jiawenzhang@fudan.edu.cn}

\date{28/08/2024}
\subjclass[2010]{}

\thanks{JZ was partly supported by National Key R{\&}D Program of China 2022YFA100700.}

\baselineskip=16pt

\begin{abstract}
In this note we study the natural question of when the generalised F\o lner sets exhibiting property A can be chosen to be subsets of the space itself. We show that for many property A spaces $X$, this is indeed possible. Specifically this holds: for all discrete bounded geometry metric spaces which coarsely have all components unbounded; for all countable discrete groups; and for all box spaces.
\end{abstract}

\maketitle

\section{Introduction}

The notion of property A is a coarse geometric property introduced by Yu in \cite{Yu00}, as a coarse analogue to the notion of amenability which implies the coarse Baum-Connes conjecture. The coarse Baum-Connes conjecture provides an algorithm to compute the $K$-theories of Roe algebras, which serve as receptacles for higher indices of elliptic differential operators on open manifolds (see \cite{Roe88}), and has fruitful applications in geometry and topology (see, \emph{e.g.}, \cite{Roe96, WY20}).

Recall that a discrete metric space $(X,d)$ has \emph{property A} if for any $R,\varepsilon>0$ there exists a family $\{A_x\}_{x\in X}$ of finite, non-empty subsets of $X \times \NN$ satisfying:
\begin{enumerate}
  \item for any $x,y\in X$ with $d(x,y) \leq R$, we have
  $\frac{|A_x \triangle A_y|}{|A_x \cap A_y|} < \varepsilon$;
  \item there exists an $S>0$ such that for any $x\in X$ and $(y,n)\in A_x$, then $d(x,y)\leq S$.
\end{enumerate}

This should be compared with the definition of amenability for countable discrete groups in terms of F\o lner sets: a group $G$ is \emph{amenable} if for any $R,\varepsilon>0$ there exists a finite, non-empty subset $A$ of $G$ satisfying:
\begin{enumerate}
  \item for any $x,y\in G$ with $d(x,y) \leq R$, we have
  $\frac{|xA \triangle yA|}{|xA \cap yA|} < \varepsilon$;
  \item there exists an $S>0$ such that for any $x\in G$ and $y\in xA$, then $d(x,y)\leq S$.
\end{enumerate}
Here $xA:=\{xg: g\in A\}$.
Note that the second condition holds vacuously in this case since $A$ is finite, however we include it here to emphasise the parallel between the definitions. It is required in the definition of property A as a substitute for equivariance.

Beyond the weakening of equivariance, there is another obvious difference in the definition of property A, namely the fact that property A allows the generalised F\o lner sets to be subsets of $X\times \NN$.  We may think of this as allowing a single point to be included multiple times in the same set. This plays a key role in allowing for non-injective maps between coarse spaces. In particular in the proof that property A is inherited by subspaces (see \cite{Wil09b, Yu00}), one needs to be able to retract neighbourhoods onto the subspace. 

A large class of groups and spaces are known to have property A (see the survey article \cite{Wil09b}), \emph{e.g.}: amenable groups, linear groups, hyperbolic spaces, CAT(0) cube complexes and spaces of finite asymptotic dimension. On the other hand, there are examples where property A fails, \emph{e.g.}: expander graphs and certain box spaces of free groups constructed by Arzhantseva, Guentner and \v{S}pakula in \cite{AGS12}.

We observe that some of these cases (\emph{e.g.}: amenable groups and hyperbolic spaces) satisfy the following simpler version of property A:

\begin{defn}\label{defn:naive property A}
A discrete metric space $(X,d)$ has \emph{naive property A} if for any $R,\varepsilon>0$ there exists a family $\{A_x\}_{x\in X}$ of finite and non-empty subsets of $X$ satisfying:
\begin{enumerate}
  \item for any $x,y\in X$ with $d(x,y) \leq R$, we have
  $\frac{|A_x \triangle A_y|}{|A_x \cap A_y|} < \varepsilon$;
  \item there exists an $S>0$ such that for any $x\in X$ and $y\in A_x$, then $d(x,y)\leq S$.
\end{enumerate}
\end{defn}

Note that the only difference between Definition \ref{defn:naive property A} and property A is that the family of sets $\{A_x\}_{x\in X}$ are required to be in $X$ rather than $X \times \NN$. While it is obvious that naive property A implies property A, it is not immediately clear how one might approach the converse.

The distinction between the definitions of property A and its naive form is one of weights and in this respect it is analogous to the relationship between $ULA$ (Uniform Local Amenability) and $ULA_\mu$ introduced in \cite{brodzki2013uniform}. The task of establishing the equivalence of property A and naive property A is thus spiritually similar to the result, proved by Elek \cite{elek2021uniform}, that $ULA$ and $ULA_\mu$ are equivalent.

\bigskip

In this paper we establish the equivalence of property A and naive property A in key cases including groups, box spaces and any coarsely connected space of bounded geometry.

To state our results in full generality, we begin with a definition. Recall that for a metric space $(X,d)$, the \emph{$r$-Rips complex of $X$} is the simplicial complex whose simplices are $\{x_0,x_1, \cdots,x_n\} \subseteq X$ with $d(x_i, x_j) \leq r$ for any $i,j=0,1,\cdots,n$.

\begin{defn}
\label{coarsely unbounded}
We say that a metric space \emph{coarsely has unbounded components} if for some $r$ every connected component of the $r$-Rips complex is unbounded.
\end{defn}

As usual, a discrete metric space $(X,d)$ is said to have \emph{bounded geometry} if the number $\sup_{x\in X} |B(x,R)|$ is finite for any $R>0$, where $B(x,R):=\{y\in X: d(y,x) \leq R\}$. 

Our main result is the following:

\begin{alphthm}\label{naive A=A}
Let $X$ be a discrete bounded geometry metric space which coarsely has unbounded components. Then $X$ has property A \emph{if and only if} $X$ has naive property A.
\end{alphthm}

The proof of Theorem A uses a new technique to construct flows in the space to redistribute mass out towards infinity, allowing us to do this whenever the space has coarsely unbounded components, without requiring any further constraints on the geometry.

\bigskip

Given that naive property A agrees with property A in such a wide range of cases it is natural to ask if naive property A is a coarse invariant. Just as with the case of subspaces, the argument for coarse invariance of property A relies on the multiplicities allowed by the definition.  However in the case of coarse equivalence for bounded geometry metric spaces we use uniform bounds on the multiplicities to demonstrate that this does preserve naive property A.

\begin{alphthm}\label{thm:coarse invariance}
If $X$ is a bounded geometry metric space with naive property A, then any bounded geometry metric space coarsely equivalent to $X$ also has naive property A.
\end{alphthm}

We do not know whether naive property A is preserved when passing to subspaces. Indeed we note the following:

\begin{alphthm}\label{thm:subspaces}
Naive property A is equivalent to property A for the class of discrete metric spaces of bounded geometry \emph{if and only if} naive property A is preserved when passing to subspaces.
\end{alphthm}

This follows from the observation that if $X$ is a discrete bounded geometry property A metric space, then $X\times \NN$ is a discrete bounded geometry property A metric space which coarsely has unbounded components, and which contains $X$ as a subspace.

\bigskip

While the class of spaces covered by Theorem \ref{naive A=A} is extremely broad, there are a number of natural examples which are not covered by this class. In particular for box spaces of residually finite groups, the components of any Rips complex are all bounded. The same is also true for infinite direct sums of finite groups when equipped with their natural coarse geometry arising from any left invariant proper metric. Nonetheless we also obtain the following results:

\begin{alphthm}\label{thm:group case}
A countable discrete group with proper left-invariant metric has property A \emph{if and only if} it has naive property A.
\end{alphthm}

\begin{alphthm}\label{thm:box space case}
Let $G$ be a residually finite group, and $G=N_0 \geq N_1 \geq N_2 \geq \cdots $ be a sequence of normal subgroups with finite index and $\bigcap_{k=0}^\infty N_k = \{1\}$. Let $X$ be the associated box space. Then $X$ has property A \emph{if and only if} it has naive property A.
\end{alphthm}

We would like to thank the referee for their helpful comments.

\section{Coarse invariance}

In this section we will prove Theorem \ref{thm:coarse invariance}.

It is clear that naive property A pushes forward under injective coarse equivalences. More precisely, if $f:X\to Y$ is such a map then given sets $A_x$ in $X$ demonstrating naive property A we can define sets $B_y\subseteq Y$ given by $f(A_x)$ where $f(x)$ is a nearest neighbour of $y$ in the image $f(X)$.

For a general coarse equivalence $f$ of bounded geometry metric spaces, pick a section $g$ of the map $f:X\to f(X)$ and set $Z=g(f(X))$.  Then $f$ yields an injective coarse equivalence from $Z$ to $Y$.  Hence it suffices to show that the inclusion of $Z$ into $X$ allows one to \emph{pull back} naive property $A$.

Let $S$ be an upper bound for the distance between a point $x$ and its image $g(f(x))$ in $Z$ (which is finite since $f$ is a coarse equivalence) and using bounded geometry let $M$ be an upper bound for the cardinality of $S$-balls in $X$. We may thus choose an injection $\iota$ from $X$ into $Z\times \{0,\dots,M-1\}$ such that $\iota(x)\in B_S(gf(x))$.

By construction, $\iota$ gives an injective coarse equivalence (where $Z\times \{0,\dots,M-1\}$ is given any sensible metric). Hence naive property $A$ for $X$ pushes forward to $Z\times \{0,\dots,M-1\}$. Therefore, the proof of Theorem \ref{thm:coarse invariance} reduces to the question of whether naive property $A$ for $Z\times \{0,\dots,M-1\}$ implies naive property $A$ for $Z$.

Thus suppose that $Z\times \{0,\dots,M-1\}$ has naive property A. Fix $R,\epsilon > 0$ and let $\epsilon'=\epsilon/M$. Let $A_{(z,j)}$ denote naive property A subsets satisfying
\[
\frac{|A_{(z,j)}\triangle A_{(z',j')}|}{|A_{(z,j)}\cap A_{(z',j')}|} < \epsilon'
\]
when the distance from $(z,j)$ to $(z',j')$ is at most $R$.

Now for $z\in Z$ set
\[
\widetilde{A}_z=\{w\in Z: A_{(z,0)}\cap (\{w\}\times \{0,\dots, M-1\})\text{ is non-empty}\}.
\]
First we note that $|\widetilde{A}_z \triangle \widetilde{A}_{z'}|\leq |A_{(z,0)}\triangle A_{(z',0)}|$, since $w\in \widetilde{A}_z \triangle \widetilde{A}_{z'}$ means that $A_{(z,0)}\cap (\{w\}\times \{0,\dots, M-1\})$ is non-empty while $A_{(z',0)}\cap (\{w\}\times \{0,\dots, M-1\})$ is empty or vice versa.

On the other hand, we claim that $M \cdot |\widetilde{A}_z \cap \widetilde{A}_{z'}|\geq |A_{(z,0)}\cap A_{(z',0)}|$. This follows as $(w,k)\in A_{(z,0)}\cap A_{(z',0)}$ implies $w\in \widetilde{A}_z \cap \widetilde{A}_{z'}$ giving a map from the former set to the latter.  This map is at most $M$ to $1$ giving the required inequality.

It now follows that
\[
\frac{|\widetilde{A}_z \triangle \widetilde{A}_{z'}|}{|\widetilde{A}_z \cap \widetilde{A}_{z'}|}<M\epsilon'=\epsilon
\]
as required which completes the proof of Theorem \ref{thm:coarse invariance}.

\section{Uniformly finite chains and non-amenable spaces}\label{sec:unif finite homology}

The notion of amenability for metric spaces was introduced by Block and Weinberger in \cite{BW92}, generalising the notion for groups. Recall that a discrete metric space $(X,d)$ is \emph{amenable} if for any $R,\varepsilon>0$, there exists a finite subset $U \subseteq X$ such that $|\partial_R U| \leq \varepsilon |U|$ where $\partial_R U:=\{x\in X \setminus U: d(x,U) \leq R\}$.

In this section, as a motivation for the general result, we prove a special case of Theorem \ref{naive A=A}: if $X$ is a discrete bounded geometry space which is non-amenable then $X$ has property A if and only if $X$ has naive property A. This will be sufficient to establish Theorem \ref{thm:group case}.

The fact that this is a special case of Theorem \ref{naive A=A} follows from the observation that if a discrete metric space $(X,d)$ does not coarsely have unbounded components, then it must be amenable.

We will use the tool of uniform finite homology established in \cite{BW92}. Let $(X,d)$ be a discrete metric space. For each $i\in \NN \cup\{0\}$, denote $X^{i+1}$ the $(i+1)$-Cartesian product of $X$ equipped with the metric:
\[
d((x_0,\cdots,x_i),(y_0,\cdots,y_i)) := \max_{0\leq j\leq i} d(x_j,y_j).
\]
Here we take the liberty of using the same notation $d$ for the metric on $X^{i+1}$ which should not cause any confusion.

Denote by $C_i^{uf}(X)$ the abelian group of \emph{infinite} formal sums
\[
c=\sum a_{\bar x} \bar x
\]
where $\bar x \in X^{i+1}$ and $a_{\bar x} \in \ZZ$, satisfying the following:
\begin{enumerate}
  \item there exists $L>0$ such that $|a_{\bar x}| \leq L$;
  \item for any $s>0$ there exists $N_s>0$ such that for any $\bar y \in X^{i+1}$, we have
          \[
          \big|\{\bar x \in B(\bar y,s) : a_{\bar x} \neq 0 \}\big| \leq N_s;
          \]
   \item there exists $r>0$ such that $a_{\bar x} =0$ if $\bar x=(x_0,\cdots,x_i)$ and $\diam\{x_0,\ldots,x_i\} > r$.
\end{enumerate}

Note that condition (2) holds trivially if $(X,d)$ has bounded geometry. 

For each $i\in \NN$, define the boundary map $\partial: C_i^{uf}(X) \rightarrow C_{i-1}^{uf}(X)$ by
\[
\partial (x_0,\cdots,x_i)=\sum_{j=0}^i (-1)^j (x_0, \cdots, \hat{x}_j, \cdots, x_i)
\]
and extend `by linearity' to the infinite sums. The pair $(C_{\bullet}^{uf}(X),\partial)$ defines a chain complex and its homology group, the \emph{uniformly finite homology} of $X$, is denoted by $H_i^{uf}(X)$.

Block and Weinberger showed that the non-amenability of $X$ is equivalent to vanishing of $H_0^{uf}(X)$. Moreover this vanishing is characterised as follows:

\begin{lem}[{\cite[Lemma 2.4]{BW92}}]
Let $(X,d)$ be a discrete metric space of bounded geometry with $H_0^{uf}(X)=0$. Then for every $x\in X$ there exists $t_x\in C_1^{uf}(X)$ such that $\partial t_x = x$ and $\sum_x t_x \in C_1^{uf}(X)$.
\end{lem}

The proof of this lemma shows that each $1$-chain $t_x$ can be taken to have the form $\sum_{j=0}^\infty (t^x_{j+1}, t^x_j)$ where the initial point $t^x_0$ is $x$. Note that as this is a uniformly finite chain in particular the distances $d(t_j^x, t_{j+1}^x)$ are uniformly bounded. We introduce the following terminology to describe the properties that these sequences enjoy:

\begin{defn}\label{defn:tails}
Let $(X,d)$ be a discrete metric space of bounded geometry. A \emph{uniform cover of $X$ by tails} is a family of sequences $\big\{t^x_j: j\in \NN \cup \{0\}\big\}_{x\in X}$ in $X$ such that $t^x_0=x$ for each $x\in X$, and there exist constants $r>0$ and $K \in \NN$ satisfying:
\begin{enumerate}
  \item for each $x\in X$, $t^x_j$ are distinct where $j\in \NN \cup \{0\}$ and $d(t_j^x, t_{j+1}^x) \leq r$;
  \item for any $z\in X$, we have $\big| \{ x\in X: z=t^x_j \mbox{~for~some~}j \in \{0\} \cup \NN\} \big| \leq K$.
\end{enumerate}
\end{defn}

Now suppose that $X$ is a discrete bounded geometry space which has property A and is non-amenable. By non-amenability it admits a uniform cover by tails for some fixed $r,K$.

Fix $R,\varepsilon>0$, and let $\varepsilon'=\frac{\varepsilon}{K}$. It is a well-known fact (see \emph{e.g.,} \cite[Theorem 1.2.4``(3)$\Rightarrow$(1)'']{Wil09b}) that $X$ having property A implies that there exists an $M>0$ and a family $\{A_x\}_{x\in X}$ of finite, non-empty subsets of $X \times \{0,1,2,\ldots, M\}$ satisfying:
\begin{itemize}
  \item for any $x,y\in X$ with $d(x,y) \leq R$, we have
  $\frac{|A_x \triangle A_y|}{|A_x \cap A_y|} < \varepsilon'$;
  \item there exists an $S>0$ such that for any $x\in X$ and $(y,n)\in A_x$, then $d(x,y)\leq S$.
\end{itemize}
Note that the constant $M$ depends on $R,\varepsilon'$ but is independent of $x\in X$.

Now we would like to transport the ``$\NN$-direction'' of $A_x$ into $X$ along the tail $t_x$. More precisely, consider the following map:
\[
\T: X \times \{0,1,2,\ldots, M\} \rightarrow X, \quad (x,n) \mapsto t^x_{M-n}.
\]
Setting $\widetilde{A}_x:=\T(A_x) \subseteq X$ for each $x\in X$, we claim that the family $\{\widetilde{A}_x\}_{x\in X}$ provides the naive property A sets for parameters $R,\varepsilon$. 

First note that for any $x\in X$ and $(y,n)\in A_x$, we have 
\[
d(t^y_{M-n},x) \leq d(t^y_{M-n}, t^y_0) + d(t^y_0,x) \leq (M-n)r + d(y,x) \leq Mr+S.
\]
Hence we obtain that $\widetilde{A}_x = \T(A_x) \subseteq B(x,Mr+S)$.

On the other hand, due to the second condition for tails we obtain that for any finite $Z\subseteq X \times \{0,1,2,\ldots, M\}$ we have $|Z| \leq K \cdot |\T(Z)|$. In particular, we have 
\[
|\widetilde{A}_x \cap \widetilde{A}_y| = |\T(A_x) \cap \T(A_y)| \geq |\T(A_x \cap A_y)| \geq \frac{1}{K}\cdot |A_x \cap A_y|
\]
for each $x,y\in X$. Also note that
\[
\widetilde{A}_x \setminus \widetilde{A}_y = \T(A_x) \setminus \T(A_y) \subseteq \T(A_x \setminus A_y)
\]
for any $x,y\in X$, hence we have
\[
|\widetilde{A}_x \triangle \widetilde{A}_y| = |\widetilde{A}_x \setminus \widetilde{A}_y| + |\widetilde{A}_y \setminus \widetilde{A}_x| \leq |A_x \setminus A_y| + |A_y \setminus A_x| = |A_x \triangle A_y|.
\]
Combining them together, we obtain that for any $x,y\in X$ with $d(x,y)\leq R$:
\[
\frac{|\widetilde{A}_x \triangle \widetilde{A}_y|}{|\widetilde{A}_x \cap \widetilde{A}_y|} \leq K \cdot \frac{|A_x \triangle A_y|}{|A_x \cap A_y|} < K\varepsilon' = \varepsilon.
\]
Hence we prove the claim, and $X$ has naive property A.

\begin{proof}[Proof of Theorem \ref{thm:group case}]
In the case of a countable discrete group with proper left-invariant metric, there is a dichotomy: either the group is amenable, in which case it has naive property A, exhibited by translating the F\o lner sets, or it is non-amenable, in which case the above argument shows that naive property A is equivalent to property A.
\end{proof}

\section{Spaces which coarsely have unbounded components}

This whole section is devoted to proving Theorem \ref{naive A=A}. The idea is similar to the non-amenable case presented in Section \ref{sec:unif finite homology}, however we cannot simply use tails to smear the property A sets. Instead, we need to transport the ``$\NN$-direction'' of the property A sets along some maximal tree. 

Recall Definition \ref{coarsely unbounded}: a discrete bounded geometry space \emph{coarsely has unbounded components} if for some $r$ every component of the $r$-Rips complex is unbounded. Given such a space we will fix $r$, and take a maximal tree in each component of the $r$-Rips complex. Bounded geometry ensures that each tree contains an embedded ray which we use to choose a point at infinity (in each component). We now orient the edges of the tree to point towards the chosen basepoint.

This yields a uniformly finite 1-chain on the original space with the following properties:

\begin{enumerate}
\item For every $x$ there is exactly one edge leaving $x$, whose other end we denote $\sigma(x)$.

\item For every $x$ the sequence $x,\sigma(x),\sigma(\sigma(x)),\dots$ is proper, that is any bounded set contains only finitely many of these points.
\end{enumerate}

We think of the map $\sigma$ as giving a flow on the space (in the place of tails for the non-amenable case) and use this to smear out the families of functions arising from property A to obtain characteristic functions exhibiting naive property A.

For a function $a$ on $X$, we denote its $\ell^1$-norm by $\|a\|_1:=\sum_{x\in X}|a(x)|$ and its support by $\supp a:=\{x\in X:a(x) \neq 0\}$. We say that $a$ is positive if $a(x) \geq 0$ for any $x \in X$.

\begin{proof}[Proof of Theorem \ref{naive A=A}]
The implication from naive property A to property A is trivial.

Supposing that $X$ has property A, for each $R,\varepsilon$ we have a family of finite, non-empty subsets $\{A_x\}_{x\in X}$ of $X\times \NN$ as specified in the definition of property A. As in Section \ref{sec:unif finite homology}, we can find an $M>0$ such that the family $\{A_x\}_{x\in X}$ can be chosen to be subsets of $X \times \{0,1,2,\ldots, M\}$.

From these we define uniformly bounded positive integer $0$-chains on $X$ by $a_x(z)=|A_x\cap (\{z\}\times \NN)|$ for $z\in X$. The conditions on the sets $A_x$ translate as:
\begin{enumerate}
  \item for any $x,y\in X$ with $d(x,y) \leq R$ we have
  \[\dfrac{\norm{a_x-a_y}_1}{\norm{a_x \land a_y}_1} < \varepsilon\]
  where $a\land b$ denotes the pointwise minimum of two chains;
  \item there exists an $S>0$ such that for any $x\in X$, $\supp a_x$ is contained in $B_S(x)$;
  \item there is a uniform bound (given by $M$ and the cardinality bound on $S$-balls) on the norm $\norm{a_x}_1$ for $x\in X$.
\end{enumerate}
Conditions $(2),(3)$ are obvious.  For (1) we note that $\norm{a_x-a_y}_1\leq |A_x \triangle A_y|$ while 
  $\norm{a_x \land a_y}_1\geq |A_x\cap A_y|$ since
\begin{align*}
\|a_x\land a_y\|_1&=\sum_z \min\{|A_x\cap (\{z\}\times \NN)|,|A_y\cap (\{z\}\times \NN)|\}
\\
&\geq\sum_z |A_x\cap (\{z\}\times \NN)\cap A_y\cap (\{z\}\times \NN)|=|A_x\cap A_y|
\\
\norm{a_x-a_y}_1&=\sum_z ||A_x\cap  (\{z\}\times \NN)|- |A_y\cap  (\{z\}\times \NN)||
\\
&\leq\sum_z |(A_x\cap  (\{z\}\times \NN))\triangle (A_y\cap  (\{z\}\times \NN))|
\\
&=\sum_z |(A_x\triangle A_y)\cap  (\{z\}\times \NN)|=|A_x\triangle A_y|.
\end{align*}

We think of values of $a(x)$ which are greater than $1$ as giving `towers' over the point $x$. The idea is to use $\sigma$ to shift the towers of $a$ thus reducing the overall height.

Given a positive integer $0$-chain $a$ on $X$, let $b(a)$, the \emph{base} of $a$, denote the characteristic function of the support of $a$. Let $t(a)$, the \emph{towers} of $a$, denote $a-b(a)$. We define $s_1(a)$ by:
\[
s_1(a)(x)=b(a)(x)+\sum_{y\in X, \sigma(y)=x} t(a)(y), \quad \text{for} \quad x\in X.
\]
We remark that $s_1$ preserves the $\ell^1$ norm (by positivity of $a$). Moreover, we have:

\begin{claim}\label{claim1}
 $s_1$ is positive in the sense that $a\leq a'$ implies $s_1(a)\leq s_1(a')$ (for $a,a'$ positive). (Here ``$\leq$'' means pointwise.)
\end{claim}

To see this, we first observe that $a\leq a'$ implies containment of the supports so $b(a)\leq b(a')$.
If $b(a)(x)=0$ then $t(a)(x)$ is also zero so $t(a)(x)\leq t(a')(x)$, while if $b(a)(x)=1$ we also have $b(a')(x)=1$ so $t(a)(x)\leq t(a')(x)$ follows from $a(x)\leq a'(x)$.
Since $b(a)\leq b(a')$ and $t(a)\leq t(a')$, it follows that $s_1(a)\leq s_1(a')$ which establishes the claim.

\bigskip

We define $s_n$ to be the $n$-fold composition of $s_1$, noting that $s_n$ is also positive.

\begin{claim}\label{claim2}
For $n$ sufficiently large, $s_n(a)$ is $0,1$-valued. The least such $n$ can be bounded depending only on the $\ell^1$-norm of $a$.
\end{claim}

Note that once $s_n(a)$ is $0,1$-valued, it follows that $s_1(s_n(a))=s_n(a)$ so the sequence is ultimately constant.

\medskip

\noindent\textit{Proof of Claim \ref{claim2}:} Note first that the support of $s_n(a)$ is a non-decreasing sequence of sets and, since $s_n$ preserves the $\ell^1$ norm, $s_n(a)$ is $0,1$-valued if and only if the support of $s_n(a)$ has cardinality $\norm{a}_1$.

Let $x \in \supp t(a)$. Consider the least $n$ such that $\sigma^n(x)$ does not lie in the support of $a$ and let $y=\sigma^n(x)$. Such an $n$ must exist by properness of the sequence $\sigma^n(x)$. Note that since $\sigma^k(x)$ lies in $\supp a$ for all $k<n$, and these points are all distinct (again by properness), it follows that $n$ is at most the cardinality of $\supp a$ which in turn is bounded by $\norm{a}_1$.

For each $k=0,1,\dots, n-1$, we have $\sigma^k(x)$ in the support of $t(s_k(a))$: this is true by definition when $k=0$ and by induction for all $k$. It follows that $y=\sigma^n(x)$ lies in the support of $s_n(a)$ (though depending on the value of $t(a)(x)$ it may not itself be a tower).

Since $y$ lies in the support of $s_n(a)$ but not the support of $a$ we deduce that the support of $s_n(a)$ is strictly greater than that of $a$. Hence
\[
\norm{s_n(a)}_1-|\supp s_n(a)| = \norm{t(s_n(a))}_1 < \norm{a}_1-|\supp a| = \norm{t(a)}_1.
\]
Hence repeating the process at most $\norm{t(a)}_1$ times reduces the difference to zero and the resulting function is $0,1$-values as required. Specifically this holds provided that $n\geq \norm{a}_1\norm{t(a)}_1$, which concludes Claim \ref{claim2}.

\bigskip

We now write $s_\infty(a)$ for the value of $s_n(a)$ when $n\geq  \norm{a}_1\norm{t(a)}_1$. 
Note that $s_\infty$ is also a positive map because if $a\leq a'$ then for $n$ sufficiently large $s_\infty(a)=s_n(a)\leq s_n(a')=s_\infty(a')$. Hence:
\begin{equation}\label{EQ:estimate}
s_\infty(a_x \land a_y) \leq s_\infty(a_x) \land s_\infty(a_y).
\end{equation}
We will write $a\setminus b:=a-(a\land b)=(a-b)\lor 0$, where $\lor$ denotes the pointwise maximum.

\begin{claim}\label{claim3}
$\norm{s_\infty(a_x) \setminus s_\infty(a_y)}_1 \leq \norm{a_x \setminus a_y}_1$.
\end{claim}

To see this we write
\begin{align*}
\norm{a_x}_1&=\norm{(a_x-a_x \land a_y)+a_x \land a_y}_1\\
&\leq\norm{a_x-a_x \land a_y}_1+\norm{a_x \land a_y}_1\\
&=\norm{a_x \setminus a_y}_1+\norm{s_\infty(a_x \land a_y)}_1\\
&\leq \norm{a_x \setminus a_y}_1+\norm{s_\infty(a_x)\land s_\infty(a_y)}_1
\end{align*}
by (\ref{EQ:estimate}). Now note that
\begin{align*}
\norm{s_\infty(a_x) \setminus s_\infty(a_y)}_1
&=\norm{s_\infty(a_x) - s_\infty(a_x)\land s_\infty(a_y)}_1\\
&=\norm{s_\infty(a_x)}_1 - \norm{s_\infty(a_x)\land s_\infty(a_y)}_1\\
&=\norm{a_x}_1 - \norm{s_\infty(a_x)\land s_\infty(a_y)}_1\\
&\leq \norm{a_x \setminus a_y}_1.
\end{align*}

\bigskip

From Claim \ref{claim3} it follows that
\[
\norm{s_\infty(a_x)-s_\infty(a_y)}=\norm{s_\infty(a_x) \setminus s_\infty(a_y)}_1+\norm{s_\infty(a_y) \setminus s_\infty(a_x)}_1\leq \norm{a_x-a_y}_1
\]
while, again by (\ref{EQ:estimate}) we have $\norm{a_x \land a_y}_1 \leq \norm{s_\infty(a_x) \land s_\infty(a_y)}_1$.
So if $d(x,y)\leq R$ then
\[
\dfrac{\norm{s_\infty(a_x)-s_\infty(a_y)}_1}{\norm{s_\infty(a_x) \land s_\infty(a_y)}_1}\leq \dfrac{\norm{a_x-a_y}_1}{\norm{a_x \land a_y}_1} < \varepsilon.
\]

Now taking the supports of the $0,1$-valued functions, $\widetilde{A}_x:=\supp s_\infty(a_x)$ for $x\in X$, we have:
\[
\frac{|\widetilde{A}_x \triangle \widetilde{A}_{y}|}{|\widetilde{A}_x \cap \widetilde{A}_{x}|}<\varepsilon.
\]
While Claim \ref{claim2}, combined with the observation that the support of $s_1(a)$ lies in the $r$-neighbourhood in $X$ of the support of $a$, ensures that each set $\widetilde{A}_x$ is supported uniformly close to $x$. Thus we have shown that $X$ has naive property A.
\end{proof}

Note that the aim in both the proof above and that given in Section \ref{sec:unif finite homology} is to achieve a uniform bound on the heights of the towers which is \emph{independent of $R,\varepsilon$}; of course for fixed $R,\varepsilon$ we have such a bound.  In Section \ref{sec:unif finite homology} this uniform bound was provided by the uniformity of the tails: the flow provided by these has the property that for any given point, there is a uniform bound on the number of other points whose towers will flow through this. Therefore in that case it is sufficient to flow by the height of the towers. By contrast in the general case there is no uniform bound on the number of towers being pushed onto any given point by the flow, therefore instead we need to continue flowing until we achieve the required uniform bound (in this case $1$) for the heights of the towers.  The time required to achieve this is bounded using the bounded geometry of the space.

\section{Spaces with bounded coarse components}

Our theorem leaves open the question of whether naive property A is equivalent to property A for spaces which are a coarse disjoint union of bounded pieces (with the distance between pieces tending to infinity). We finish by considering a couple of natural examples arising from groups.

Consider first the example of $\bigoplus_{n\in\NN} \ZZ/2$ with proper metric.  This is not covered by Theorem \ref{naive A=A} since every connected component of each Rips complex is bounded.  However it is of course covered by Theorem \ref{thm:group case}. Indeed it is amenable so has naive property A.

Secondly we consider box spaces associated to residually finite groups with a given tower of normal subgroups. Again Theorem \ref{naive A=A} does not apply, however in these examples the equivalence of property A and naive property A is again true:

\begin{proof}[Proof of Theorem \ref{thm:box space case}]
We only need to show the sufficiency. It is well-known that $X$ has property A if and only if $G$ is amenable. Hence for any $\varepsilon>0$ and $R>0$, there exists a finite subset $F \subseteq G$ such that $|gF \triangle hF| < \varepsilon |gF \cap hF|$ when $d_G(g,h)\leq R$.

Set $S = \max\{\diam(F), d_G(e,F)\}$. By the assumption, we know that there exists $J \in \mathbb{N}$ such that for any $j>J$ the quotient map $\pi_j: G \to G/N_j$ is an isometry restricted to any ball with radius at most $R+2S$. Enlarging $J$ if necessary, we can also assume that $d(G/N_j, G/N_l) > R$ for any $j,l > J$.

Now for $\bar{g} \in G/N_j$, we set
\[
A_{\bar{g}}:=
\begin{cases}
\bar{g}\cdot \pi_j(F)& \text{ if } j>J\\
\bigsqcup_{i=1}^J G/N_i& \text{ if }j\leq J.
\end{cases}
\]
We show that these $A_{\bar{g}}$ provide naive property A sets for $R,\varepsilon$.

First note that for any $g\in G$ we have $d_G(g, gF) = d_G(e,F) \leq S$, which implies that $d(\pi_j(g), \pi_j(gF)) \leq S$ for any $j$. On the other hand, for any $\bar{g}, \bar{h} \in G/N_j$ with $d(\bar{g}, \bar{h}) \leq R$ for some $j>J$, we can choose their representatives $g,h \in G$ satisfying $d_G(g,h) \leq R$. Note that $\diam(gF \cup hF) \leq R+2S$, hence the quotient map $\pi_j$ is an isometry on $gF \cup hF$. Therefore, we obtain:
\[
|A_{\bar{g}} \triangle A_{\bar{h}}| = |\pi_j(gF) \triangle \pi_j(hF)| = |\pi_j(gF \triangle hF)| = |gF \triangle hF|,
\]
and similarly,
\[
|A_{\bar{g}} \cap A_{\bar{h}}| = |gF \cap hF|.
\]
Combining these, we complete the proof.
\end{proof}

\bibliographystyle{plain}
\bibliography{NaiveA}

\end{document}